\documentclass[11pt]{amsart}

\usepackage{amsmath}
\usepackage{amssymb}

\usepackage{amsfonts}
\usepackage{amsthm}
\usepackage{enumerate}
\usepackage[mathscr]{eucal}

\usepackage{tikz}

\usepackage{bbm}
\usepackage{subfigure}

\usepackage{graphicx}
\usepackage{verbatim}

\newcommand{\Be}{\begin{equation}}
\newcommand{\Ee}{\end{equation}}

\newcommand{\Bea}{\begin{eqnarray}}
\newcommand{\Eea}{\end{eqnarray}}
\newcommand{\Beas}{\begin{eqnarray*}}
\newcommand{\Eeas}{\end{eqnarray*}}
\newcommand{\Benu}{\begin{enumerate}}
\newcommand{\Eenu}{\end{enumerate}}
\newcommand{\Bi}{\begin{itemize}}
\newcommand{\Ei}{\end{itemize}}

\def\intslash{\rlap{\kern  .32em $\mspace {.5mu}\backslash$ }\int}
\def\qsl{{\rlap{\kern  .32em $\mspace {.5mu}\backslash$ }\int_{Q_x}}}

\def\emph#1{{\it #1 }}

\def\dist{{\text{\it dist}}}

\def\supp{{\text{\rm supp}}}

\def\inn#1#2{\langle#1,#2\rangle}

\def\card{\text{\rm card}}
\def\lc{\lesssim}

\def\ep{\epsilon}

             \def\La{\Lambda}

\def\fM{{\mathfrak {M}}}

\def\bbE{{\mathbb {E}}}

\def\bbN{{\mathbb {N}}}

\def\bbR{{\mathbb {R}}}

\def\bbZ{{\mathbb {Z}}}

\def\sD{{\mathscr {D}}}

\def\sH{{\mathscr {H}}}

\def\cA{{\mathcal {A}}}

\def\cS{{\mathcal {S}}}

\def\cU{{\mathcal {U}}}

\def\tf{{\widetilde f}}

 at 10 true pt

\def\be#1{\begin{equation}\label{ #1}}
\def\endeq{\end{equation}}
\def\endal{\end{align}}
\def\bas{\begin{align*}}
\def\eas{\end{align*}}
\def\bi{\begin{itemize}}
\def\ei{\end{itemize}}

\def\emph#1{{\it #1}}
\def\textbf#1{{\bf #1}}

\usepackage{bbm}
\def\bbone{{\mathbbm 1}}

\theoremstyle{plain}
  \newtheorem{theorem}{Theorem}[section]

   \newtheorem{proposition}[theorem]{Proposition}
   
   \newtheorem{lemma}[theorem]{Lemma}
   \newtheorem{corollary}[theorem]{Corollary}

\theoremstyle{remark}

\theoremstyle{definition}

\newcommand {\SE} {{\mathbb E}}
\newcommand {\SN} {{\mathbb N}}
\newcommand {\SR} {{\mathbb R}}

\newcommand {\SNz} {{\mathbb \SN_0}}
\newcommand {\SRd} {{\mathbb R^d}}
\newcommand {\SZd} {{\mathbb Z^d}}

\newcommand {\e} {{\varepsilon}}

\newcommand{\bfe}{{\boldsymbol\e}}

\newcommand{\ENp}{{\bbE_N^\perp}}

\newcommand {\mand} {{\quad\mbox{and}\quad}}
\renewcommand {\mid} {{\,\,\,\colon\,\,\,}}
\def\supp{\mathop{\rm supp}}
\def\dist{\mathop{\rm dist}}

\newcounter{reg}
\newcommand{\sline}{{\smallskip

\noindent}}

\def\Xint#1{\mathchoice
{\XXint\displaystyle\textstyle{#1}}%
{\XXint\textstyle\scriptstyle{#1}}%
{\XXint\scriptstyle\scriptscriptstyle{#1}}%
{\XXint\scriptscriptstyle\scriptscriptstyle{#1}}%
\!\int}
\def\XXint#1#2#3{{\setbox0=\hbox{$#1{#2#3}{\int}$ }
\vcenter{\hbox{$#2#3$ }}\kern-.6\wd0}}

\def\mint{\Xint-}

\begin{document}

\title
[The Haar system  as a Schauder basis]
{The Haar system as a Schauder basis in spaces of Hardy-Sobolev type}

\author[G. Garrig\'os \ \ \ A. Seeger \ \ \ T. Ullrich] {Gustavo Garrig\'os   \ \ \ \   Andreas Seeger \ \ \ \ 
Tino Ullrich}

\address{Gustavo Garrig\'os\\ 
Department of Mathematics\\
University of Murcia\\
30100 Espinardo\\Murcia, 
Spain
} 
\email{gustavo.garrigos@um.es}

\address{Andreas Seeger \\ Department of Mathematics \\ University of Wisconsin \\480 Lincoln Drive\\ Madison, WI,
53706, USA} \email{seeger@math.wisc.edu}

\address{Tino Ullrich\\
Hausdorff Center for Mathematics\\ Endenicher Allee 62\\
53115 Bonn, Germany} \email{tino.ullrich@hcm.uni-bonn.de}
\begin{abstract} We show that, for suitable enumerations, the Haar system is a Schauder basis in the classical Sobolev spaces in $\SR^d$ with integrability $1<p<\infty$ and 
smoothness $1/p-1<s<1/p$. This complements earlier work by the last two authors on the unconditionality of the Haar system and implies  that it is a {conditional} Schauder basis for a nonempty open subset  of  the $(1/p,s)$-diagram.  
The results extend to (quasi-)Banach spaces of Hardy-Sobolev and Triebel-Lizorkin type in the  range of parameters $\frac{d}{d+1}<p<\infty$ and $\max\{d(1/p-1),1/p-1\}<s<\min\{1,1/p\}$,  which is optimal except perhaps at the end-points. 
\end{abstract}

\thanks{G.G. was supported in part   by grants  MTM2013-40945-P, MTM2014-57838-C2-1-P, MTM2016-76566-P
from MINECO (Spain), and grant 19368/PI/14  from Fundaci\'on S\'eneca (Regi\'on
de Murcia, Spain). A.S. was supported in part by NSF grant 
DMS 1500162. T.U. was supported 
the DFG  Emmy-Noether program UL403/1-1}

\maketitle



\section{Introduction}

We recall the definition of the (inhomogeneous) Haar system in $\SRd$.
Consider the 1-variable functions
\[
h^{(0)}=\bbone_{[0,1)}\mand h^{(1)}=\bbone_{[0,1/2)}-\;\bbone_{[1/2,1)}.
\]For every $\bfe=(\e_1,\ldots,\e_d)\in\{0,1\}^d$ one defines
\[
h^{(\bfe)}(x_1,\ldots,x_d)\,=\,h^{(\e_1)}(x_1)\cdots h^{(\e_d)}(x_d).
\]
Finally, one sets
\[
h^{(\bfe)}_{k,\ell}(x)= h^{(\bfe)}(2^kx-\ell),\quad  k\in\SNz,\;\ell\in\SZd,
\]
Denoting $\Upsilon=\{0,1\}^d\setminus\{\vec 0\}$, the Haar system is then given by
\[
\sH_d=\Big\{h^{(\vec 0)}_{0,\ell}\Big\}_{\ell\in\SZd}\cup \Big\{h^{(\bfe)}_{k,\ell}\mid k\in\SNz,\;\ell\in\SZd,\;\bfe\in\Upsilon\Big\}.
\]
Observe that $\supp\;h^{(\bfe)}_{k,\ell}$ is the dyadic cube $I_{k,\ell}:=2^{-k}(\ell+[0,1]^d)$.

In this paper we consider basis properties of $\sH_d$ in  
Besov spaces 
$B^s_{p,q}$, and Triebel-Lizorkin spaces
$F^s_{p,q}$ in $\SRd$. We refer to \cite{albiac-kalton} for terminology and general facts about bases in Banach spaces. 
 
In the 1970's, Triebel \cite{triebel73,triebel78} proved that the Haar system $\sH_d$ is a   Schauder
 basis on $B^s_{p,q}(\SRd)$ if 
\begin{equation}\label{large_range}
 \tfrac{d}{d+1}<p<\infty,\quad0<q<\infty,\quad 
\max\big\{d(\tfrac1p-1),\tfrac1p-1\big\} <s< \min\big\{1,\tfrac1p\big\},
\end{equation}
 and this range is maximal, except perhaps at the endpoints.
 Moreover, the basis is unconditional when  \eqref{large_range} holds; see \cite[Theorem 2.21]{triebel-bases}.
Concerning $F^s_{p,q}$ spaces, however, in \cite{triebel-bases} it is only shown  that $\sH_d$ is an unconditional basis for $F^s_{p,q}(\SRd)$ when, besides \eqref{large_range}, the additional assumption 
\Be \max\big\{d(\tfrac1q-1),\tfrac1q-1\big\}<s<\tfrac1q\label{dqs}\Ee is satisfied.  Recently, two of the authors showed in
\cite{su,sudet} that the additional restriction \eqref{dqs} is in fact  necessary,  at least when $d=1$.  It was left open 
whether suitable enumerations of the Haar system
can form  a Schauder basis in $F^s_{p,q}$ in  the larger range \eqref{large_range}.  We shall answer this question affirmatively.
 
Given an  enumeration $\{u_1,u_2,\ldots\}$ of the system $\sH_d$,
we let $P_N$  be the orthogonal projection onto the subspace spanned by $u_1,\dots, u_N$, i.e. 
\Be \label{PNdef} P_N f= \sum_{n=1}^N \|u_n\|_2^{-2} \inn{f}{u_n}u_n \,.
\Ee

The sequence $\{u_n\}_{n=1}^\infty$ is a Schauder basis on $F^s_{p,q}$ if 
\Be \label{limit} \lim_{N\to\infty} \|P_Nf-f\|_{F^s_{p,q}} =0,\quad \mbox{for all }f\in F^s_{p,q}.\Ee
In view of the uniform boundedness principle, density theorems and the result for Besov spaces,  \eqref{limit} follows if we can show that the operators $P_N$ have 
uniform $F^s_{p,q}\to F^s_{p,q}$ operator norms. Note, that the condition $s<1/p$ is necessary since the Haar functions need to belong to 
$F^s_{p,q}$. By duality, if  $1<p<\infty$, the condition $s>1/p-1$ becomes also necessary, 
so the range in \eqref{large_range} is optimal in this case.  If $p\leq 1$, 
then an interpolation argument shows that \eqref{large_range} is also a maximal range, except perhaps at the end-points; see $\S\ref{optimal}$ below. 

\smallskip 

\noindent{\bf Definition.}
An enumeration $\cU=\{u_1, u_2, ...\}$ of the Haar system $\sH_d$ is {\it admissible}
if the following condition holds for each cube $I_{\nu} = \nu+[0,1]^d, \nu\in \bbZ^d$.
If $u_n$ and $u_{n'}$ are both supported in $I_\nu$ and $|\supp(u_n)|>|\supp(u_{n'})|$, then necessarily $n<n'$\,.

\begin{figure}[h]
    \centering
    $\begin{tabular}{c|ccccccc}
    $k \backslash I_\nu$ & $I_{\nu_0}$ & $I_{\nu_1}$ & $I_{\nu_2}$ & $I_{\nu_3}$ & $I_{\nu_4}$ & \ldots \\
    \hline
        0&1&2&4&7&11\\
        1&3&5&8&12  \\
            2&6&9&13\\
            3&10&14\\
            4&15
    \end{tabular}$
    \caption{An admissible enumeration of $\sH_d$.}
\end{figure}

\noindent The table above shows how to obtain an admissible (natural) enumeration of $\sH_d$ via a diagonalization of the intervals $I_{\nu}$ versus the levels 
$k$. 
We first label the set $\SZd=\{\nu_1,\nu_2,\ldots\}$. 
Then, we follow the order indicated by the table, where being at position $(\nu_i,k)$ means to pick all the Haar functions with support contained in $I_{\nu_i}$ and size $2^{-kd}$, arbitrarily enumerated,
before going to the subsequent table entry. 

Our main result reads as follows. 

\begin{theorem} \label{schauder}
Let  
$\cU=\{u_n\}_{n=1}^\infty$ be an admissible enumeration of the Haar system $\sH_d$. 
Assume that

(i) $\frac{d}{d+1}<p<\infty$, 

(ii) $0<q <\infty$, 

(iii) $\max\{d(\frac 1p-1),\frac 1p-1\} <s< \min\{1,\frac 1p\}$. 

Then $\cU$ is a Schauder basis on $F^s_{p,q}(\SRd)$.
\end{theorem}

\begin{figure}[h]
 \centering
\subfigure
{\begin{tikzpicture}[scale=2]
\draw[->] (-0.1,0.0) -- (2.1,0.0) node[right] {$\frac{1}{p}$};
\draw[->] (0.0,-1.1) -- (0.0,1.1) node[above] {$s$};

\draw (1.0,0.03) -- (1.0,-0.03) node [below] {$1$};
\draw (2.0,0.03) -- (2.0,-0.03) node [below] {$2$};
\draw (1.5,0.03) -- (1.5,-0.03) node [below] {$\frac{3}{2}$};
\draw (0.03,1.0) -- (-0.03,1.00) node [left] {$1$};
\draw (0.03,.5) -- (-0.03,.5) node [left] {$\frac{1}{2}$};
\draw (0.03,-.5) -- (-0.03,-.5) node [left] {$-\frac{1}{2}$};
\draw (0.03,-1.0) -- (-0.03,-1.00) node [left] {$-1$};

\draw[dashed] (1.0,0.0) -- (1.0,1.0);
\draw[fill=black!70, opacity=0.4] (0.0,-.5) -- (0.0,0.0) -- (.5,.5) -- (1.5,0.5)
-- (1.0,0.0) -- (.5,-.5) -- (0.0,-.5);
\draw (0.0,-1.0) -- (0.0,0.0) -- (1.0,1.0) -- (2.0,1.0) -- (1.0,0.0) --
(0.0,-1.0);
\end{tikzpicture}
}
\subfigure
{\begin{tikzpicture}[scale=2]
\draw[->] (-0.1,0.0) -- (2.1,0.0) node[right] {$\frac{1}{p}$};
\draw[->] (0.0,-1.1) -- (0.0,1.1) node[above] {$s$};

\draw (1.0,0.03) -- (1.0,-0.03) node [below] {$1$};
\draw (1.5,0.03) -- (1.5,-0.03) node [below] {{\tiny$\;\;\frac{d+1}{d}$}};
\draw (1.25,0.03) -- (1.25,-0.03) node [below] {{\tiny $\!\frac{2d+1}{2d}$}};
\draw (0.03,1.0) -- (-0.03,1.00) node [left] {$1$};
\draw (0.03,.5) -- (-0.03,.5) node [left] {$\frac{1}{2}$};
\draw (0.03,-.5) -- (-0.03,-.5) node [left] {$-\frac{1}{2}$};
\draw (0.03,-1.0) -- (-0.03,-1.00) node [left] {$-1$};

\draw[dashed] (1.0,0.0) -- (1.0,1.0);
\draw[fill=black!70, opacity=0.4] (0.0,-.5) -- (0.0,0.0) -- (.5,.5) -- (1.25,0.5)
-- (1.0,0.0) -- (.5,-.5) -- (0.0,-.5);
\draw (0.0,-1.0) -- (0.0,0.0) -- (1.0,1.0) -- (1.5,1.0) -- (1.0,0.0) --
(0.0,-1.0);
\end{tikzpicture}
}
\caption{Unconditionality of the Haar system in Hardy-Sobolev spaces in $\SR$ and $\SRd$}\label{fig2}
\end{figure}
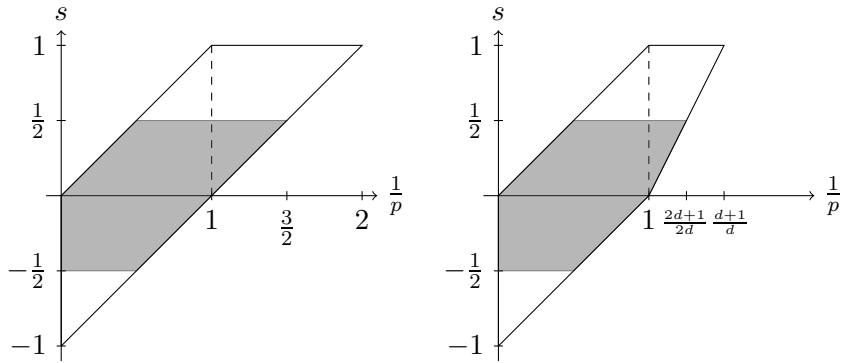

In the left part of Figure \ref{fig2}, the trapezoid is the parameter domain for which the Haar system is a 
Schauder basis in 
 the Hardy-Sobolev space $H^s_p(\mathbb{R})$ ($= F^s_{p,2} (\bbR)$) while the 
shaded part represents the parameter domain for which the Haar system is an unconditional basis in $H^s_p(\mathbb{R})$. The right 
figure shows the respective parameter domain for $H^s_p(\mathbb{R}^d)$.

The heart of the matter is a boundedness result for the dyadic averaging operators $\bbE_N$ given by  
\Be
\label{condexp}
\bbE_N f(x)= \sum_{\mu\in \bbZ^d} \bbone_{I_{N,\mu}} (x) \, 2^N \int_{I_{N,\mu}} f(t) dt\, 
\Ee
with $$I_{N,\mu}=2^{-N}(\mu+[0,1)^d),\quad \mu\in\SZd,\;N=0,1,2,\ldots$$ 
Note that $\bbE_N f$  is just the conditional expectation of $f$ 
with respect to the $\sigma$-algebra generated by the set $\sD_N$ of all dyadic cubes of length $2^{-N}$. 
There is  a well known 
 relation 
between the Haar system 
and the dyadic averaging
 operators, namely for 
$N=0,1,2,\dots$,
\Be\label{martdiff} 
\bbE_{N+1} f-\bbE_Nf
= \sum_{\bfe\in\Upsilon}\sum_{\mu\in \bbZ^d} 2^{Nd} \inn{f}{ h^{(\bfe)}_{N,\mu} }h^{(\bfe)}_{N,\mu},
\Ee
i.e.  $\bbE_{N+1} -\bbE_N $ is the orthogonal projection onto the space generated by the Haar functions with Haar frequency $2^N$.

Now let $\eta_0$ be a Schwartz function on $\bbR^d$, supported in
$\{|\xi|<3/8\} $ and so that
$\eta_0(\xi)=1$ for $|\xi|\leq 1/4$.
Let $\varPi_N$ be defined by
\Be\label{defofvarPi}\widehat{ \varPi_N f}(\xi)= \eta_0(2^{-N}\xi)\widehat f(\xi). \Ee
There is  a 
basic standard 
inequality (almost immediate from the definition of Triebel-Lizorkin spaces)
\Be\label{Pinineq}\sup_N\|\varPi_N f\|_{F^s_{p,q}}\le C(p,q,s)\|f\|_{F^s_{p,q}}\Ee
which is valid for all $s\in \bbR$  and for $0<p<\infty$, $0<q\le \infty$. Moreover, 
\eqref{Pinineq} and the fact that
$\|\varPi_Ng-g\|_{F^s_{p,q}}\to 0$ for Schwartz functions $g$ 
gives 
\Be\label{approfid}\lim_{N\to\infty}  \|\varPi_Nf-f\|_{F^s_{p,q}}
= 0\Ee if $f\in F^s_{p,q}$ and $0<p,q<\infty$.
The main tool in proving Theorem \ref{schauder} is a similar bound for the operators $\bbE_N$ which of course follows from the corresponding bound for $\bbE_N-\varPi_N$. It turns out that
the operators $\bbE_N-\varPi_N$ enjoy better mapping properties in Besov spaces. 
 
Similar bounds are also satisfied by projection operators into sets of Haar functions with fixed Haar frequency. 
Namely,  for  $N\in\bbN$  and functions
$a\in \ell^\infty(\bbZ^d\times\Upsilon)$, we define
\Be \label{TNdef}T_{N}[f,a]
=
\sum_{\bfe\in\Upsilon}\sum_{\mu\in \bbZ^d} a_{\mu,\bfe} 2^{Nd}\inn{f}{h^{(\bfe)}_{N,\mu}} h^{(\bfe)}_{N,\mu}.
\Ee
Observe that the choice $a_{\mu,\bfe}\equiv 1$
recovers the operator $\bbE_{N+1}-\bbE_N$. Then, we shall prove the following.

\begin{theorem}\label{expthm} Let ${d}/{(d+1)}<p\le\infty$, $0<r\leq\infty$, and 
\Be\max\{d(1/p-1),1/p-1\} <s< \min\{1,1/p\}.\label{large2}\Ee
Then there is  a constant $C:=C(p,r,s)>0$ such that for
 all $f\in B^s_{p,\infty}$
\Be \label{Besovbd}
\sup_N    \|\bbE_Nf -\varPi_Nf \|_{B^s_{p,r}} \leq C\|f\|_{B^s_{p,\infty}}.
\Ee
Moreover,
\Be \label{TNbd}\sup_N \| T_N[f,a]\|_{B^s_{p,r}}\lc   \|a\|_\infty \|f\|_{B^s_{p,\infty}}\,. \Ee

\end{theorem}

We have  the embedding $F^s_{p,q}\subset F^s_{p,\infty}\subset B^s_{p,\infty}$ which we use on the function side. For $r\le p$  we have the embedding  $B^s_{p,r}\subset F^s_{p,r}$ (by Minkowski's inequality in $L^{p/r}$) and if also $r<q$ we have $F^s_{p,r}\subset F^s_{p,q}$; these two are used  for  $\bbE_N f-\varPi_N f$, or $T_N[f,a]$. In particular  we conclude from Theorem \ref{expthm}  that $\bbE_N-\varPi_N$ is bounded on $F^s_{p,q}$, uniformly in $N$.  Hence
\begin{corollary}\label{uniformbdcor}
 Let $p,s$ be as in \eqref{large2} and  $0<q\leq\infty$. Then 
\Be \label{uniformbd}
\sup_N    \|\bbE_Nf \|_{F^s_{p,q}}+ \sup_N\sup_{\|a\|_{\ell^\infty}\le 1} \|T_N[f,a]\|_{F^s_{p,q}} \lc \|f\|_{F^s_{p,q}}.
\Ee
\end{corollary}



The proofs in this paper use basic principles in the theory of function spaces, such as $L^p$  inequalities for the Peetre maximal functions. 
A different approach to Corollary \ref{uniformbdcor}
via wavelet theory is  presented  in  the subsequent paper \cite{gsu-wavelet}. 
The main arguments and the proof of Theorem \ref{expthm} are contained in
\S \ref{expthmsect}.  In \S\ref{schaudersect} we show how  estimates in the proof of Theorem \ref{expthm} are  used to deduce Theorem \ref{schauder}.  Finally, in \S\ref{optimal} we discuss the optimality of the results.

\section{Proof of Theorem \ref{expthm}}\label{expthmsect}

We start with some preliminaries on convolution kernels which are used  in Littlewood-Paley type decompositions.
Let $\beta_0, \beta $ be Schwartz functions on $\SRd$, compactly supported in $(-1/2,1/2)^d$ such that
 $|\widehat{\beta}_0(\xi)|>0$ when  $|\xi|\leq1$ and 
 $|\widehat{\beta}(\xi)|>0$ when $1/8\leq|\xi|\leq1$. Moreover assume $\beta$ has  vanishing
moments up to a large order 
\Be\label {Mlowbound}M> \frac dp +|s|,
\Ee
 that is,
\Be\label{betacanc}
    \int_{\SRd} \beta(x)\,x_1^{m_1}\cdots x_d^{m_d}\,dx = 0\quad\mbox{when}\quad m_1+\ldots+m_d < M\,.
\Ee
For $k=1,2,\dots$ let  $\beta_k:=2^{kd}\beta(2^k\cdot)$ and $L_k f=\beta_k*f$.   
We shall use the inequality 
\Be\label{localmeans} 
    \|g\|_{B^s_{p,r}}\lc  \Big(\sum_{k=0}^\infty 2^{ksr}\|L_k g\|_p^r\Big)^{1/r}
    \Ee
    and apply it to $g=\bbE_Nf-\varPi_N f$. 
    Inequality 
   \eqref{localmeans} is of course just one part of a characterization of $B^s_{p,r}$ spaces by sequences of compactly supported kernels  
   (or  `local means'), 
   with sufficient cancellation assumptions, see for example \cite[\S
2.5.3]{triebel2}.

Let $\eta_0\in C^\infty_c(\SRd)$ be as in  \eqref{defofvarPi}, that is, supported on $\{|\xi|<3/8\}$ and such that $\eta_0(\xi)=1$ when $|\xi|\leq1/4$. Define $\La_0$, and $\La_k$ for $k\ge 1$ by
\begin{align*}
\widehat{\La_0 f}(\xi) &=\frac{\eta_0(\xi)}{\widehat \beta_0(\xi)}\widehat f(\xi)
\\
\widehat{\La_k f}(\xi) &=\frac{\eta_0(2^{-k}\xi) -\eta_0(2^{-k+1}\xi)}{\widehat \beta(2^{-k}\xi)}\widehat f(\xi), \quad k\ge 1.
\end{align*}
Then $\sum_{j=0}^\infty L_j \La_j=\text{\sl Id}$ with convergence in $\cS'$, and
\[\sup_{j\ge 0}2^{js}\|\La_j f\|_p\lc \|f\|_{B^s_{p,\infty}}\,.
\]
 Moreover $\varPi_N = \sum_{j=0}^N L_j\La_j$, and therefore
\Be\label{resolotionofvarPi}
\bbE_N f-\Pi_N f= \sum_{j=0}^N( \bbE_N L_j\La_j f - L_j\La_j f) +\sum_{j={N+1}}^\infty \bbE_N L_j\La_j f.
\Ee 
If we use the convenient notation\[
\ENp := I-\bbE_N,
\]
then the asserted estimate \eqref{Besovbd} will follow from
\Be\label{j>N}\Big(\sum_{k=0}^\infty 2^{ksr} \Big\| \sum_{j=N+1}^\infty  L_k  \bbE_N L_j\La_j f
\Big\|_p^r\Big)^{1/r} \lc 
\sup_j 2^{js}\|\Lambda_j f\|_p\,.
\Ee
and
\Be\label{j<N}
\Big(\sum_{k=0}^\infty 2^{ksr} \Big\| \sum_{j=0}^N  L_k  
\ENp L_j\La_j f
\Big\|_p^r\Big)^{1/r} \lc 
\sup_j 2^{js}\|\Lambda_j f\|_p\,.
\Ee 

Below we shall use variants of the Peetre maximal functions, which are a standard tool in the study of Besov and Triebel-Lizorkin spaces. 
 We define
\begin{subequations}\label{peetredef}
\begin{align}\fM_j g(x)&= \sup_{|h|_\infty\le 2^{-j+1}} |g(x+h)|\,,
\\
\fM_j^* g(x)&= \sup_{|h|_\infty\le 2^{-j+5}} |g(x+h)|\,,
\\
 \label{peetrenontang}
\fM_{A,j}^{**} g(x)&= \sup_{h\in \SRd}\frac{ |g(x+h)|}{(1+2^j|h|)^A}\,,
\end{align}
\end{subequations} 
where $|h|_\infty=\max\{|h_1|,\ldots,|h_d|\}$, $h=(h_1,\ldots,h_d)\in\SRd$.
These different versions are introduced for technical purposes in the proofs. They  satisfy  obvious  pointwise inequalities,
\[\fM_j g(x)\le 
\fM_j^* g(x) \le C_A \fM_{A,j}^{**} g(x),\] 
and 
\Be \label{maxcomp}\begin{aligned}\fM_j g(x)& \le \inf_{|h|_\infty\le 2^{-j+4}} \fM_j^* g(x+h)
\\&\le \Big(2^{(j-4)d} \int_{|h|_\infty\le 2^{-j+4}}
 [\fM_j ^*g(x+h)]^r dh\Big)^{1/r},\quad 0<r\leq\infty.\end{aligned}\Ee

Below we shall use Peetre's inequality (\cite{Pe}, see also \cite[$\S1.3.1$]{Tr83})
\Be \label{peetreLp}
\|\fM_{A,j}^{**} f\|_p\le C_{p,A}
\|f\|_p ,  \text{ $\quad 0<p\leq\infty$,  
 $\quad A>d/p$, }
 \Ee 
 for $f\in \cS'(\bbR^d)$  satisfying
\Be\label{specassu}
 \supp (\widehat f)
 \subset \{\xi: |\xi|\le 2^{j+1}\} .
 \Ee
 Throughout  we shall assume that $M\gg A$;
we require  specifically 
 $$
d/p<A< M-|s|
\,.$$


 The main estimates needed in the proof of \eqref{j>N} 
 and  \eqref{j<N} 
 are summarized in
 
 \begin{proposition} \label{LkENLjprop} Let $0<p\leq \infty$ and
 \Be \label{BjkN}
 B(j,k,N)= \begin{cases} 
 2^{N-j}\,2^{\frac{j-k}p}\,
 2^{(j-N)(d-1)(\frac1p-1)_+}  &\text{ if } j,k\ge N+1,
 \\
 2^{\frac{N-k}p} 2^{j-N}  &\text{ if } j\le N, \,\, k\ge N+1,
 \\
  2^{k-N} 2^{j-N}2^{(N-k)d(\frac1p-1)_+} &\text{ if } 0\le j,k\le N,
  \\
  2^{k-j+\frac{j-N}p+[N-k+(j-k)(d-1)](\frac1p-1)_+} &\text{ if } j\ge N+1, \,\, k\le N.
 \end{cases} 
 \Ee
 
 Then the following inequalities hold for all
 $f\in \cS'(\bbR^d)$ 
 whose Fourier transform is supported in
 $\{|\xi|\le 2^{j+1}\}$.
  
  (i) For $j\ge N+1$,
  \Be \label{LkENLjgeN}
  \|L_k\bbE_N[L_j f]\|_p \lc\begin{cases}
   B(j,k,N) \|f\|_p &\text{ if } k\ge N+1,
   \\
  [B(j,k,N)+ 2^{-|j-k|(M-A)}]\|f\|_p &\text{ if } 0\le k\le N.
  \end{cases}
  \Ee
  
  (ii) For $0\le j\le N$,
 \Be 
\label{LkENLjleN}
  \|L_k\ENp [L_j f]\|_p \lc \begin{cases}  \big[B(j,k,N)+ 2^{-|j-k|(M-A)}\big] \|f\|_p
  &\text{ if } k\ge N+1,
  \\ B(j,k,N) \|f\|_p
  &\text{ if } 0\le k\le N.
  \end{cases}
\Ee





(iii) The same bounds hold if the operators $\bbE_N$ in (i) and 
$\bbE_N^\perp$ in (ii) are replaced by $T_N[\cdot,a]$, uniformly in $\|a\|_\infty\le 1$.
\end{proposition}




 
 We begin with two preliminary lemmata, the first a straightforward estimate for $L_kL_j$.

\begin{lemma} \label{LjLk} Let $k,j\ge 0$  and suppose that $f$ is locally integrable.
Let $M$ be as in \eqref{betacanc} with $M>A>d/p$.
Then 
\Be\label{LjLkpt} |L_k L_j f(x) |\lc 2^{-|k-j|(M-A)} \fM_{A,\max\{j,k\}}^{**} f(x).\Ee
If $f\in\cS'(\SRd)$ with
$\widehat f(\xi)=0$ for $|\xi|\ge 2^{j+1}$ then
\[\|L_kL_j f\|_p  \lc 2^{-|k-j|(M-A)} \|f\|_p\,.\]
\end{lemma}
\begin{proof}
The second assertion is an immediate consequence of 
\eqref{LjLkpt}, by  \eqref{peetreLp}. 
We have $L_kL_jf=\gamma_{j,k}*f$ where $\gamma_{j,k}=\beta_k*\beta_j$. By symmetry we may assume $k\le j$. Using the cancellation assumption \eqref{betacanc} on the $\beta_j$
we get 

\begin{align*} &|\gamma_{j,k}(x)|= \Big|\int 2^{kd} \Big[ \beta(2^k(x-y) -\sum_{m=0}^{M-1} 
 \frac{1}{m!}\inn{-2^ky}{\nabla}^m \!\beta(2^kx)\Big] 
2^{jd} \beta(2^jy) dy\Big|
\\&= \Big|\int 2^{kd} 
 \int_0^1
  \frac{(1-s)^{M-1}}{(M-1)!} \inn{-2^ky}{\nabla}^M\!\beta(2^kx-s2^ky)\, ds\, 
2^{jd} \beta(2^jy) dy\Big|
\\&\lc 2^{(k-j)M}\,2^{kd} \,\bbone_{[-2^{-k},2^{-k}]^d}(x),
\end{align*} 

and thus
\begin{align*} &2^{(j-k)M} |\gamma_{j,k}*f(x) |\lc 2^{kd}
 \int_{|h|_\infty\le 2^{-{k}}}|f(x-h)|\, dh
\\
&\lc 2^{kd}\int_{|h|_\infty\le 2^{-k}} \frac{2^{(j-k)A}|f(x-h)|}{(1+2^j|h|)^A} dh \lc 2^{(j-k)A}\,\fM_{A,j}^{**} f(x).
\end{align*}
Hence \eqref{LjLkpt} holds.
\end{proof}

\subsection*{\it Some notation} (i) Below, when $j>N$  we use the notation
\[\cU_{N,j} =\Big\{(y_1,\ldots,y_d)\in\SRd\mid \min_{1\leq i\leq d}\dist (y_i, 2^{-N}\bbZ)\le 2^{-j-1}\Big\}.\]
That is, $\cU_{N,j}$ is a $2^{-j-1}$-neighborhood of the set $\cup_{I\in\sD_N}\;\partial I$.

(ii) For a dyadic cube $I$  of side  length  $2^{-N}$ and $j>N$ we shall  denote by 
$\sD_l[\partial I]$ the set of dyadic cubes $J\in \sD_{l}$ such that 
$\bar{J}\cap\partial I\not=\emptyset$.

(iii) For a dyadic cube $I$ of side length $2^{-N}$ denote by $\sD_N(I)$ the neighboring cubes of $I$, that is, the cubes $I'\in\sD_N$ with $\bar{I}\cap \bar{I'}\not=\emptyset$. 

\medskip

\begin{lemma} \label{supportlemma}
(i) Let $k>N\ge 1$ and $g$ be locally integrable. Then
\Be L_k (\bbE_N g)(x)=0, \quad\text{ for all } x\in \cU_{N,k}^\complement=\bbR^d\setminus\cU_{N,k}\,.\Ee

(ii) Let $j>N\ge 1$, and $f$ locally integrable. Then 
\Be \bbE_N [L_jf]=\bbE_N[L_j(\bbone_{\cU_{N,j}} f)].\label{Lemma2.i}
\Ee
Moreover, 
\Be
\big|\SE_N(L_jf)\big|\lesssim 2^{(N-j)d}\sum_{I\in\sD_N} \sum_{J\in\sD_{j+1}[\partial I]}\|f\|_{L^\infty(J)}\,\bbone_I.
\label{ENjmax}
\Ee
\end{lemma}
\begin{proof} (i) We use the support and cancellation properties of $\beta_k$.
Note that \[
L_k(\SE_N g)(x)=\int \beta_k(x-y)\,\SE_Ng(y)\,dy,
\]
and $\supp\beta_k(x-\cdot)\subset x+2^{-k}[-1/2,1/2]^d$.
So, if $I\in\sD_N$  and $x\in I\cap \cU_{N,k}^\complement$,
then $\supp\beta_k(x-\cdot)\subset I$, and hence 
\[
L_k(\SE_N g)(x)=(\SE_Ng)_{|_I}(x)\,\int_I \beta_k(x-y)\,dy\,=\,0.
\]
 (ii) One argues similarly. First note that, changing the order of integration, \Be
\SE_N(L_jf)=\sum_{I\in\sD_N}\int_{\SRd}f(y)\Big[\int_I\beta_j(x-y)\,dx\Big]\;dy\;\frac{\bbone_I}{|I|}.\label{ENjaux}\Ee Now if 
$J\in\sD_N$ and $y\in J\cap \cU_{N,k}^\complement$ then $\supp\beta_j(\cdot-y)
\subset J$, and hence $\int_I\beta_j(x-y)\,dx=0$. Thus $\SE_N[L_j(\bbone_{\cU_{N,j}^\complement}f)]=0$. Finally, to prove \eqref{ENjmax} note that, if $I\in\sD_N$ and $x\in I$, 
then from \eqref{ENjaux} it follows  
\Beas
\big|\SE_N(L_jf)(x)\big|& =& |I|^{-1}\;\Big|\sum_{J\in\sD_{j+1}[\partial I]} \int_{J}f(y)\Big[\int_I\beta_j(x-y)\,dx\Big]\;dy\Big|\;\\
& \leq & 2^{Nd}\sum_{J\in\sD_{j+1}[\partial I]}\|f\|_{L^\infty(J)}
2^{-(j+1)d}\|\beta_j\|_1,
\Eeas
which gives the asserted \eqref{ENjmax}.
\end{proof}

\subsection*{\bf Proof of Proposition \ref{LkENLjprop}}

 \subsubsection*{\it 
 Proof of \eqref{LkENLjgeN} in the case $j,k\ge N+1$}
 
 By Lemma \ref{supportlemma}.i, $L_k\bbE_N[L_j f](x)=0$ if $x\in\cU_{N,K}^\complement$, so we assume that $x\in \cU_{N,k}\cap I$, for some $I\in\sD_N$. Recall that $\sD_N(I)$ consists of  the neighboring cubes of $I$.
 Then \eqref{ENjmax} and the support property of $\beta_k$ give
 \Beas
|L_k\bbE_N[L_j f](x)|& \leq & \int|\beta_k(x-y)|\,\big|\SE_N(L_jf)(y)\big|\,dy\\
& \lc &   2^{(N-j)d} \sum_{I'\in\sD_N(I)}\sum_{J\in\sD_{j+1}[\partial I']}
\|f\|_{L^\infty(J)}\,\|\beta_k\|_{1}.
\Eeas
Hence
\Bea
\|L_k\bbE_N[L_j f]\|_p & =&  \Big[\sum_{I\in\sD_N}\int_{I\cap\cU_{N,k}}|L_k(\SE_N L_jf)|^p\,dx\Big]^{\frac1p}\nonumber\\
& \lesssim &  2^{(N-j)d}\Big[\sum_{I\in\sD_N}
\Big(\sum_{J\in\sD_{j+1}[\partial I]}\|f\|_{L^\infty(J)}\Big)^p 
\,|I\cap\cU_{N,k}|\Big]^\frac1p.\label{jkN_aux1}
\Eea
Now, $|I\cap\cU_{N,k}|\approx 2^{-k}2^{-N(d-1)}$, and $\card \,\sD_{j+1}[\partial I]\approx 2^{(j-N)(d-1)}$. Also, if we write $J=2^{-j-1}(\ell_J+[0,1]^d)$, then
\[
\|f\|_{L^\infty(J)}\leq \inf_{|h|_\infty\leq 2^{-j-1}}\fM^*_jf(\ell_J+h)
\leq \Big[2^{jd}\int_{|h|_\infty\leq 2^{-j-1}}\fM^*_jf(\ell_J+h)^p\,dh\Big]^{\frac1p}.\]
Therefore, using either H\"older's inequality (if $p>1$), or the embedding $\ell^p\hookrightarrow\ell^1$ (if $p\leq1$), we have 
\begin{align}
 & \Big[\sum_{I\in\sD_N}
\Big(\!\!\!\sum_{J\in\sD_{j+1}[\partial I]}\|f\|_{L^\infty(J)}\Big)^p \Big]^\frac1p  \nonumber
\\ & \lesssim  2^{(j-N)(d-1)(1-\frac1p)_+}\Big[\sum_{I\in\sD_N}\sum_{J\in\sD_{j+1}[\partial I]} \|f\|_{L^\infty(J)}^p\Big]^\frac1p\nonumber\\
& \lesssim  
2^{(j-N)(d-1)(1-\frac1p)_+}\,\Big[\sum_{J\in\sD_{j+1}} 2^{jd}\int_{|h|_\infty\leq 2^{-j-1}}\fM^*_jf(\ell_J+h)^p\,dh\Big]^\frac1p\nonumber\\
& \lesssim  
2^{(j-N)(d-1)(1-\frac1p)_+}\, 2^{\frac{jd}p}\,\big\|\fM^*_jf\big\|_{L^p(\SRd)}\,.\label{jkN_aux2}
\end{align}
Finally, inserting \eqref{jkN_aux2} into \eqref{jkN_aux1}, and using \eqref{peetreLp}, yields
\Beas
\|L_k\bbE_N[L_j f]\|_p &\lesssim& 2^{(N-j)d} 2^{(j-N)(d-1)(1-\frac1p)_+}\, 2^{\frac{jd}p}\,\|f\|_p\,
2^{-\frac kp}2^{-\frac{N(d-1)}p}\\
& = & 2^{N-j}\,2^{\frac{j-k}p}\,
 2^{(j-N)(d-1)(\frac1p-1)_+}\,\|f\|_p\,,
\Eeas
using in the last step the trivial identity $(1-\frac1p)_+=(\frac1p-1)_+-(\frac1p-1)$.
This establishes 
\eqref{LkENLjgeN} for $j,k\ge N+1$.
	\qed
     \subsubsection*{\it Proof of \eqref{LkENLjleN} in the  case $j\le N$, $k\ge N+1$.}
For $w\in I$ with $I\in\sD_N$ we have
  \begin{align*}
  &\big|\SE_N^\perp(L_jf)(w)\big|=\big|\bbE_N[L_j f](w)-L_jf(w)\big|\\
	&\quad=2^{Nd}\Big|\int_{I} \int_{\SRd} 2^{jd}\big[\beta(2^j(v-y))-\beta(2^j(w-y))\big] f(y) dy\, dv\Big|
  \\
  &\quad= 2^{(N+j)d} \Big|
\int_{I} 
\int_{\SRd}  \int_0^1 \nabla\beta(2^j[(1-t)w+tv-y])\cdot 2^j(v-w)\,dt\,f(y)\, dy \, dv  \Big|
\\&
\quad\leq  2^{(N+j)d} 2^{j-N} \,\int_I\,\int_0^1\,\int_{\SRd}|f(y)|\,
|\nabla\beta(2^j[(1-t)w+tv-y])|\,dy\,dt\,dv
 \\ &\quad\lc 2^{j-N} \,\fM_jf(w) ,
 \end{align*} 
since for fixed $w,v,t$ the expression involving $\nabla\beta$ is supported 
in the set $\{y\mid|y-w|_\infty\leq 2^{-j-1}+2^{-N}\}$. Moreover, since $k>N$, when  $w\in I_{N,\mu}$ and $|z|_\infty\leq 2^{-k-1}$ we have
  \Be \label{ptwisediff} 
   \big|\bbE_N^\perp[L_jf](w-z)\big| \lc 2^{j-N} \inf_{|h|_\infty\le 2^{-j}} \fM_j^* f(2^{-N}\mu+h)\,,
\Ee
and therefore,\Bea
\big|L_k\big(\SE_N^\perp[L_jf]\big)(w)\big| & \leq & \int\,|\SE_N^\perp(w-z)|\,|\beta_k(z)|\,dz\nonumber\\
& \lc &  2^{j-N} \Big[\mint_{|h|_\infty\le 2^{-j}}\!\! \fM_j^* f(2^{-N}\mu+h)^p\,dh\Big]^\frac1p.\label{kN_aux1}
\Eea
Now Lemma \ref{supportlemma}.i gives 
\begin{multline}
\big\|L_k\big(\SE_N^\perp[L_jf]\big)\big\|_p\\ \lesssim
\|L_kL_jf\|_{L^p(\cU_{N,k}^\complement)}+
\Big[\sum_{\mu\in\SZd}
\|L_k\big(\SE_N^\perp[L_jf]\big)\big\|_{L^p(\cU_{N,k}\cap I_{N,\mu})}^p
\Big]^\frac1p.
\label{kN_aux2}\end{multline}
Using \eqref{kN_aux1}, the last term is controlled by
\Beas
2^{j-N} \Big[\sum_{\mu\in\SZd}|I_{N,\mu}\cap\cU_{N,k}| \;\mint_{|h|_\infty\le 2^{-j}}\!\! \fM_j^* f(2^{-N}\mu+h)^p\,dh\Big]^\frac1p &  \\
\quad \lesssim 2^{j-N} \,[2^{-k}2^{-N(d-1)}]^\frac1p\,
2^{\frac {Nd}p}\,\| \fM_j^* f\|_p\;\lesssim\;2^{j-N}\,2^{\frac{N-k}p}\,\|f\|_p.  
\Eeas
Finally, the first term in \eqref{kN_aux2} is controlled by Lemma \ref{LjLk}, so overall one obtains\[
\big\|L_k\big(\SE_N^\perp[L_jf]\big)\big\|_p\lesssim [2^{-(M-A)|k-j|}+2^{j-N}\,2^{\frac{N-k}p}]\,\|f\|_p,
\]
establishing 
 \eqref{LkENLjleN} in the  case $j\le N$, $k\ge N+1$.
 \qed
  
  \subsubsection*{
  \it Proof of \eqref{LkENLjleN} in the case $0\le j,k\le N$.}
  We use  $$\int_{I} \bbE_N^\perp[ L_jf](y) \,dy=0, \quad I\in\sD_N,$$ to write
  \[L_k\big(  \bbE_N^\perp[ L_jf]\big)(x)=\sum_\mu \int_{I_{N,\mu} }   \big(\beta_k(x-y) -\beta_k(x-2^{-N}\mu)\big)\,
 \bbE_N^\perp[ L_jf](y)\, dy\,.
\]
 For fixed $x$, we say that 
 \Be\label{Lax} \text{$\mu\in\La(x)$ if $ |x-2^{-N}\mu|_\infty\le 2^{-N}+2^{-k-1} $}.\Ee Observe that only these $\mu$'s contribute to the above sum. 
Notice also that
 \[|\beta_k(x-y) -\beta_k(x-2^{-N}\mu)|\lc 2^{kd}\,2^{k-N},\quad \text{ if $y\in I_{N,\mu}$,}\]
 and since $j\le N$,  the estimate in \eqref{ptwisediff} gives
\[
 \big|\bbE_N^\perp[L_jf](y)\big| \lc 2^{j-N} \inf_{|h|_\infty\le 2^{-j}} \fM_j^* f(2^{-N}\mu+h)\,,\quad y\in I_{N,\mu}.
\] 
Combining all these bounds we obtain
\begin{align*}
|L_k\big(  \bbE_N^\perp[ L_jf]\big)(x)|
 \,\lc\, 2^{(k-N)(d+1)} 2^{j-N} 
 \sum_{\mu\in\La(x)}\Big(\mint_{\frac\mu{2^N}+[-\frac1{2^j},\frac1{2^j}]^d}\!\! [\fM_j^* f]^p\,\Big)^\frac1p
 &
 \\
\quad\lesssim 2^{(k-N)(d+1)} 2^{j-N}2^{(N-k)d(1-\frac1p)_+}\, 
 \Big(\sum_{\mu\in\La(x)}\mint_{\frac\mu{2^N}+[-\frac1{2^j},\frac1{2^j}]^d}\!\! [\fM_j^* f]^p\Big)^\frac1p,
&
\end{align*} using in the last step H\"older's inequality (or $\ell^p\hookrightarrow\ell^1$ if $p\leq1$) and the fact that $\card\,\La(x)\approx 2^{(N-k)d}$. 
 Observe also that the $L^p$-quasinorm of the last bracketed expression satisfies
\Beas
\Big(\int_{\SRd}\sum_{\mu\in\La(x)}\mint_{\frac\mu{2^N}+[-\frac1{2^j},\frac1{2^j}]^d}\!\! [\fM_j^* f]^p\,\Big)^\frac1p & \approx& \Big(\sum_{\mu\in\SZd}2^{-kd}\,\mint_{\frac\mu{2^N}+[-\frac1{2^j},\frac1{2^j}]^d}\!\! [\fM_j^* f]^p\,\Big)^\frac1p\\
& \lc & 2^{(N-k)d/p}\,\big\|\fM_j^*f\|_p.
\Eeas
Thus, overall we obtain
\Beas
\big\|L_k\bbE_N^\perp[ L_jf]\big\|_p
& \lesssim & 2^{(k-N)(d+1)} 2^{j-N}2^{(N-k)d(1-\frac1p)_+}\,2^{(N-k)d/p}\,\|f\|_p\\
& = & 2^{k-N} 2^{j-N}2^{(N-k)d(\frac1p-1)_+}\,\|f\|_p,\Eeas
after simplifying the indices in the last step. This establishes
\eqref{LkENLjleN} in the case $0\le j,k\le N$.
  \qed

    \subsubsection*{\it Proof of \eqref{LkENLjgeN} in the case $j\ge N+1$, $k\le N$.}
 
This
 condition and \eqref{Lemma2.i} in Lemma \ref{supportlemma} imply that $\SE_N[L_jf]=\SE_N[L_j(f\bbone_{\cU_{N,j}})]$. For simplicity, we denote $\tf= f\bbone_{\cU_{N,j}}$, and write
   \Be
	L_k\bbE_N[L_j f] 
   = L_k(\bbE_N[L_j \tf] -L_j \tf) + L_kL_j \tf.
	\label{jN_eq}\Ee
	Observe that, by Lemma \ref{LjLk},
	\[
	\|L_kL_j\tf\|_p\lesssim 2^{-(M-A)|j-k|}\,\big\|\fM^{**}_{A,j}f(x)\big\|_p\lesssim 2^{-(M-A)|j-k|}\,\|f\|_p.
	\]
So, we only need to estimate  $\|L_k\SE_N^\perp[ L_j\tf]\|_p$. Proceeding as in the proof 
of the case $j,k\le N$,
we write 
(with $\La(x)$ as in \eqref{Lax})
\begin{align}
&|L_k\big(  \bbE_N^\perp[ L_j\tf]\big)(x)|
\nonumber
\\& \quad\leq  \sum_{\mu\in\La(x)} \int_{I_{N,\mu} }   \big|\beta_k(x-y) -\beta_k(x-2^{-N}\mu)\big|\,
 \big|\SE_N^\perp[ L_j\tf](y)\big|\, dy\,\nonumber\\
& \quad\lc  2^{kd}2^{k-N}\sum_{\mu\in\La(x)} \int_{I_{N,\mu} }\Big(|\SE_N[L_j\tf]|\,+\;|L_j(\tf)|\Big)\nonumber\\
& \quad= \cA_1(x)+\cA_2(x).\label{AxBx}
\end{align}
      Now, using again \eqref{ENjmax}, we have
\Bea
|\cA_1(x)| & \lc & 		2^{(k-N)(d+1)}2^{(N-j)d}\sum_{\mu\in\La(x)}	
\sum_{J\in\sD_{j+1}[\partial I_{N,\mu}]}\|f\|_{L^\infty(J)}\nonumber\\
& \lc & 		2^{k-N}2^{(k-j)d}\,2^{(N-k)d(1-\frac1p)_+}\,\Big[\sum_{\mu\in\La(x)}	
\big(\sum_{J\in\sD_{j+1}[\partial I_{N,\mu}]}\|f\|_{L^\infty(J)}\big)^p\Big]^\frac1p,
\label{A_aux1}\Eea
since $\card\,\La(x)\approx 2^{(N-k)d}$. Taking the $L^p$-quasinorm of the last bracketed expression gives
\begin{align}
&\Big[\int_{x\in\SRd}\sum_{\mu\in\La(x)}	
\big(\!\!\!\sum_{J\in\sD_{j+1}[\partial I_{N,\mu}]}\!\!\|f\|_{L^\infty(J)}\big)^p\,dx\Big]^\frac1p
\notag
\\
&\quad\quad\lc \Big[\sum_{I\in\sD_N} 2^{-kd}\big(\!\!\sum_{J\in\sD_{j+1}[\partial I]}\|f\|_{L^\infty(J)}\big)^p\Big]^\frac1p
\notag
\\
&
\quad\quad
\lc \; 2^{\frac{(j-k)d}p}\,2^{(j-N)(d-1)(1-\frac1p)_+}\, \big\|\fM^*_jf\big\|_{L^p(\SRd)}\, \quad \mbox{{\footnotesize by \eqref{jkN_aux2}}}.
\label{A_aux2}\end{align}
Therefore, combining exponents in \eqref{A_aux1} and \eqref{A_aux2} one obtains
\Bea
\|\cA_1\|_p & \lc & 
2^{k-N}2^{(k-j)d}\,2^{(N-k)d(1-\frac1p)_+}\,
2^{\frac{(j-k)d}p}\,2^{(j-N)(d-1)(1-\frac1p)_+}\,\|f\|_p\nonumber\\
& = & 2^{k-j}\,2^{\frac{j-N}p}\,2^{(N-k)(\frac1p-1)_+}\,2^{(j-k)(d-1)(\frac1p-1)_+}\,\|f\|_p.\label{A_aux3}
\Eea		
Finally, we estimate the term $\cA_2(x)$ in \eqref{AxBx}. First notice that
\[
|L_j(\tf)(y)|\leq\int_{\cU_{N,{j}}}|\beta_j(y-z)||f(z)|\,dz=0, \quad \text{if } y\in\cU_{N,{j-1}}^\complement,
\]      
since $\supp\beta_j(y-\cdot)\subset y+2^{-j}[-\frac12,\frac12]^d
\subset \cU_{N,{j}}^\complement$.		
Moreover, if $I\in\sD_N$, then for every cube $J\in\sD_j$ such that $J\subset I\cap\cU_{N,j-1}$ we have
\[
|L_j(\tf)(y)|\leq\int|\beta_j(z)||f(y-z)|\,dz\lesssim\|f\|_{L^\infty(J^*)},
\quad \text{if } y\in J
\]
where $J^*=J+2^{-j}[-\frac12,\frac12]^d$. Therefore,
\[
\int_I|L_j(\tf)(y)|\lesssim\sum_{J\in\sD_{j}[\partial I]}\|f\|_{L^\infty(J^*)}|J|,
\]
and overall we obtain\[
|\cA_2(x)|\lesssim 2^{(k-j)d}2^{k-N}\,\sum_{\mu\in\La(x)}
\sum_{J\in\sD_{j-1}[\partial I_{N,\mu}]}\|f\|_{L^\infty(J)}.
\]
But this is essentially the same expression we obtained in \eqref{A_aux1} for the term $|\cA_1(x)|$, so the same argument will give an estimate of 
$\|\cA_2\|_p$ in terms of the quantity in \eqref{A_aux3}.
This concludes the proof of 
 \eqref{LkENLjgeN} in the case $j\ge N+1$, $k\le N$.

Finally, concerning (iii) in Proposition \ref{LkENLjprop},
we remark that the previous proofs can easily be adapted
replacing the operators $\SE_N$ and $\SE_N^\perp$ by $T_N[\cdot,a]$,
 keeping in mind that $T_N[g,a]$ is now constant in cubes $I\in\sD_{N+1}$, and enjoys an additional cancellation, $\int_{I_{N,\mu}}T_N[g,a](x) dx=0$, which simplifies some of the previous steps. \qed

  \subsection*{Proof of Theorem \ref{expthm}, conclusion}
   It remains to prove inequalities \eqref{j>N} and \eqref{j<N}.
 By  the embedding properties for the sequence spaces $\ell^r$ it suffices to verify both inequalities for very small $r$, say
  \Be \label{r-reduction} r\le \min \{p,1\}.\Ee
  In view of the embedding $\ell^r\hookrightarrow \ell^1$ and Minkowski's inequality (in $L^{p/r}$)  it suffices then to prove
\Be \label{firststar}\sup_N\Big(\sum_{k=0}^\infty 2^{ksr} \sum_{j=N+1}^\infty \big\| 
   L_k  \bbE_N L_j\La_j f
\big\|_p^r\Big)^{1/r} \lc \sup_{j}2^{js} \|\Lambda_j f\|_p
\Ee
and
\Be\label{secondstar}\sup_N
\Big(\sum_{k=0}^\infty 2^{ksr} \sum_{j=0}^N \big\| 
 L_k  (\bbE_N^\perp L_j\La_j f )\big\|_p^r\Big)^{1/r} \lc 
\sup_{j}2^{js} \|\Lambda_j f\|_p\,.
\Ee

If we apply Proposition \ref{LkENLjprop} to  each of the functions $\La_jf$, we reduce matters to observe  that 
\Be
\sup_N \sum_{k=0}^\infty 2^{ksr} \sum_{j=0}^\infty \big[2^{-js}B(j,k,N)\big]^r< \infty, 
\label{CjkN}\Ee
with  $B(j,k,N)$ as in \eqref{BjkN}, and that  
\[ \Big(\sum_{j=N+1}^\infty \sum_{k=0}^N + \sum_{k=N+1}^\infty \sum_{j=0}^N\Big)  2^{-|j-k|(M-A)} <\infty \]
which is trivial. 
The verification of \eqref{CjkN} under the assumptions in \eqref{large2} is also elementary, but we  carry out some details to clarify
how the conditions on $p$ and $s$ are used.


When $j,k>N$, we have  $B(j,k,N)=2^{N-j}\,2^{\frac{j-k}p}\,
 2^{(j-N)(d-1)(\frac1p-1)_+}$ and thus
\begin{align} \notag &\sum_{k>N}2^{ksr} \sum_{j>N}\big[2^{-js}B(j,k,N)\big]^r
\\&= \Big(\sum_{k>N}2^{-kr(\frac1p-s)}\Big)\Big(\sum_{j>N}2^{-rj[s+1-\frac1p-(d-1)(\frac1p-1)_+]}\Big)\,2^{Nr[1-(d-1)(\frac1p-1)_+]},\label{sumCjkN} 
\end{align}
and the series converge provided $s<1/p$ and \Be
s>\tfrac1p-1 + (d-1)(\tfrac1p-1)_+\,=\, \max\Big\{d(\tfrac1p-1),\tfrac1p-1\Big\}.
\label{sgtr}\Ee
Further, being geometric sums, the final outcome in \eqref{sumCjkN} is bounded uniformly in $N$.

Next assume $j\leq N<k$, then  $B(j,k,N)=2^{\frac{N-k}p} 2^{j-N}$ and hence
\[
\sum_{k>N}2^{ksr} \sum_{j\leq N}\big[2^{-js}B(j,k,N)\big]^r = \Big(\sum_{k>N}2^{-kr(\frac1p-s)}\Big)\Big(\sum_{j\leq N}2^{rj(1-s)}\Big)\,2^{Nr(\frac1p-1)},\]
which are finite expressions provided $s<\min\{1,1/p\}$.

Consider $j,k\leq N$, with  $B(j,k,N)=2^{k-N} 2^{j-N}2^{(N-k)d(\frac1p-1)_+}$. Then
\[
\sum_{k\leq N}2^{ksr} \sum_{j\leq N}\big[2^{-js}B(j,k,N)\big]^r=\hskip7cm\]
\[ = \Big(\sum_{k\leq N}2^{kr[s+1-d(\frac1p-1)_+]}\Big)\Big(\sum_{j\leq N}2^{rj(1-s)}\Big)\,2^{-Nr[2-d(\frac1p-1)_+]},\]
which lead to uniform expressions in $N$ under the assumptions $s<1$ and 
\Be s>d(\tfrac1p-1)_+-1,\label{sgtr2}\Ee
the latter being weaker than \eqref{sgtr}.

When $k\leq N<j$ we have  $B(j,k,N)=2^{k-j+\frac{j-N}p+[N-k+(j-k)(d-1)](\frac1p-1)_+}$ and
\[
\sum_{k\leq N}2^{ksr} \sum_{j> N}\big[2^{-js}B(j,k,N)\big]^r=\hskip7cm\]
\[ = \Big(\sum_{k\leq N}2^{kr[s+1-d(\frac1p-1)_+]}\Big)\Big(\sum_{j> N}
2^{-rj[s+1-\frac1p-(d-1)(\frac1p-1)_+]}\Big)\,2^{-Nr[\frac1p-(\frac1p-1)_+]},\]
where in the first series we would use \eqref{sgtr2} and in the second series \eqref{sgtr}.
We have verified \eqref{CjkN} in all cases. This finishes the proof of Theorem \ref{expthm}. \qed

\section{Schauder bases}\label{schaudersect}
Let $P_N$ be defined as in \eqref{PNdef}. For the proof of Theorem \ref{schauder} we need to prove that $\|P_Nf-f\|_{F^s_{p,q}}\to 0$ for every $f\in F^s_{p,q}$,
with $(p,s)$ as in \eqref{large2} and $0<q<\infty$. We first discuss some preliminaries about localization and pointwise multiplication by characteristic functions of cubes, then prove uniform bounds for the $F^s_{p,q}\to F^s_{p,q} $ operator norms of the $P_N$ and then establish the asserted limiting property.

\subsection*{\it Preliminaries} 
For $\nu\in \bbZ^d$ let $\chi_\nu$ be the characteristic function of $\nu+[0,1)^d$.
\begin{lemma}\label{loc}
Assume that \Be\tfrac{d-1}d<p<\infty,\quad 0<q\leq \infty,\mand\max\{d(\tfrac1p-1),\tfrac1p-1\}<s<\tfrac1p.\label{range_mult}\Ee
Then, the following holds for all $g_\nu$ and $f\in F^s_{p,q}$: 
$$\Big\|\sum_{\nu \in \bbZ^d} \chi_\nu g_\nu\Big\|_{F^s_{p,q}} \lc \Big(\sum_{\nu \in \bbZ^d} 
\big\| g_\nu\big\|_{F^s_{p,q}}^p\Big)^{1/p}$$
and
$$\Big(\sum_{\nu \in \bbZ^d} 
\big\| f\chi _\nu\big\|_{F^s_{p,q}}^p\Big)^{1/p}\lc \|f\|_{F^s_{p,q}}\,.
$$
\end{lemma} 
\begin{proof} Let $\varsigma\in C^\infty_c(\bbR^d)$ so that $\supp(\varsigma)\subset(-1,1)^d$ and 
$\sum_{\nu\in \SZd} \varsigma(x-\nu)=1$ for all $x\in \bbR^d$. Let $\varsigma_\nu=\varsigma(\cdot-\nu)$. We have, for all $s\in \bbR$,
\begin{equation}\label{2.4.7}
  \|g\|_{F^s_{p,q}} \asymp \Big(\sum_\nu \big\|  \varsigma_\nu g\big\|_{F^s_{p,q}}^p\Big)^{1/p}\,;
\end{equation}
see \cite[2.4.7]{triebel2}. Hence 
\begin{align*}
&\Big\|\sum_{\nu\in\SZd} \chi_\nu g_\nu\Big\|_{F^s_{p,q}} =
\Big\|\sum_{\nu'}\varsigma_{\nu'} \sum_\nu  \chi_\nu g_\nu\Big\|_{F^s_{p,q}}
 \lc \Big(\sum_{\nu'} 
 \Big\|\varsigma_{\nu'}\!\!\sum_{|\nu-\nu'|_\infty\leq1}   \chi_\nu g_\nu\Big\|_{F^s_{p,q}}^p \Big)^{1/p}
 \\
& \lc \Big(\sum_{\nu'} \sum_{|\nu-\nu'|_\infty\leq1}   
 \|
 g_\nu\|_{F^s_{p,q}}^p \Big)^{1/p}\lc
 \Big( \sum_{\nu}   
 \|
 g_\nu\|_{F^s_{p,q}}^p \Big)^{1/p}.
 \end{align*}
Here we have used that $ \varsigma_{\nu'}  \chi_\nu $ are pointwise multipliers of $F^s_{p,q}$, with uniform bounds in $(\nu, \nu')$, in the range given by \eqref{range_mult};
see \cite[Thm.\ 4.6.3/1]{RuSi96}. This proves the first inequality. 

For the second inequality we first observe that, by \eqref{2.4.7},
$$
   \|f\chi_{\nu}\|_{F^s_{p,q}} \lesssim \Big(\sum\limits_{\nu'} \|f\chi_{\nu} \varsigma_{\nu'}\|^p_{F^s_{p,q}}\Big)^{1/p}\,,\quad \nu \in \bbZ^d,
$$
which yields 
\begin{equation}\nonumber
 \begin{split}
     \Big(\sum\limits_{\nu}\|f\chi_{\nu}\|_{F^s_{p,q}}^p)^{1/p} &\lesssim \Big(\sum\limits_{\nu}\sum\limits_{\nu'} \|f\chi_{\nu} \varsigma_{\nu'}\|^p_{F^s_{p,q}}\Big)^{1/p}\\
     &\lesssim \Big(\sum\limits_{\nu'}\sum\limits_{|\nu-\nu'|_\infty\leq 1} \|f\chi_{\nu} \varsigma_{\nu'}\|^p_{F^s_{p,q}}\Big)^{1/p}\\
     &\lesssim \Big(\sum\limits_{\nu'}\|f\varsigma_{\nu'}\|^p_{F^s_{p,q}}\Big)^{1/p}\lesssim \|f\|_{F^s_{p,q}}\,,
 \end{split}
\end{equation}
where we have used 
 the pointwise multiplier assertion \cite[Thm.\ 4.6.3/1]{RuSi96} and then again 
\eqref{2.4.7}  in the last step.
\end{proof}

\subsection*{\it Uniform boundedness of the $P_N$} 
Observe that by the localization property of the Haar functions we have
$P_Nf = \sum_{\nu\in\SZd}  \chi_\nu P_N f= \sum_{\nu} \chi_\nu P_N [f\chi_\nu].$
Thus by Lemma \ref{loc} 
$$
\|P_N f\|_{F^s_{p,q}} \lc 
\Big(\sum_\nu 
\big\| P_N[f\chi_\nu]\big\|_{F^s_{p,q}}^p\Big)^{1/p}\,.$$
Since the enumeration of the Haar system is assumed to be admissible we have
\Be \label{PNdec}
P_N[f\chi_\nu]= \bbE_{N_\nu} [f\chi_\nu]+ T_{N_\nu}[f\chi_\nu, a^{N,\nu}]
\Ee
for some $N_\nu\in \bbN$, with $N_\nu\le N$ and appropriate  sequences $a^{N,\nu}$ assuming only the values $1$ and $0$. We remark that for each $\nu$, $N_\nu=N_\nu(N)$ with
\Be\label{limitNnu} \lim_{N\to\infty} N_\nu(N)=\infty\,.\Ee

By Theorem \ref{expthm} 
\begin{align*}
&\Big(\sum_\nu 
\big\| P_N[f\chi_\nu]\big\|_{F^s_{p,q}}^p\Big)^{1/p}
\\ 
\lc
&\Big(\sum_\nu 
\big\|    \bbE_{N_\nu} [f\chi_\nu]       \big\|_{F^s_{p,q}}^p\Big)^{1/p}
+\Big(\sum_\nu 
\big\|    T_{N_\nu}[f\chi_\nu, a^{N,\nu}]     \big\|_{F^s_{p,q}}^p\Big)^{1/p}
\\
 \lc &
\Big(\sum_\nu 
\big\| f\chi_\nu\big\|_{F^s_{p,q}}^p\Big)^{1/p} \,\lc \, \|f\|_{F^s_{p,q}}\,,
\end{align*}
where for the last inequality we have used 
 Lemma \ref{loc} again.

\begin{proof}[Proof of Theorem \ref{schauder}, conclusion]
Let $f\in F^s_{p,q}$, with $(p,s)$ as in \eqref{large2} and $0<q<\infty$.  
Let $C=\max\{ 1, \sup_N\|P_N\|_{F^s_{p,q}\to F^s_{p,q}} \}.$
Since Schwartz functions are dense in $F^s_{p,q}$ when $0<p,q<\infty$  there is $\tilde f\in \cS(\bbR)$ such that 
$\|f-\tilde f\|_{F^s_{p,q}}<(3C)^{-1}\ep$ and hence $\|P_N f-P_N \tilde f\|_{F^s_{p,q}}< \ep/3$. Choose $s_1$ so that $s<s_1<\max\{1/p,1\}$ then 
$\tilde{f} \in B^{s_1}_{p,q} \hookrightarrow F^s_{p,q}$. Since the Haar system is an unconditional basis on $B^{s_1}_{p,q}$ (\cite{triebel-bases}) we have
$\lim_{N\to\infty} \|P_N\tilde f-\tilde f\|_{B^{s_1}_{p,q}}=0$ and therefore $\lim_{N\to\infty} \|P_N\tilde f-\tilde f\|_{F^{s}_{p,q}}=0$. Combining these facts we get $\|P_Nf -f\|_{F^s_{p,q}}<\ep$ for sufficiently large $N$ which shows that 
 $P_N f\to f$ in $F^s_{p,q}$\,.
 \end{proof}

\section{Optimality away from the end-points}
\label{optimal}

\begin{proposition}\label{P4}
Let $0<q<\infty$. Then, the Haar system $\sH_d$ is not a Schauder basis of $F^s_{p,q}(\SRd)$ in each of the following cases:

\sline (i) if $1<p<\infty$ and $s\geq 1/p$ or $s\leq 1/p-1$,

\sline (ii) if $d/(d+1)\leq p\leq 1$ and $s>1$ or $s< d(1/p-1)$,

\sline (iii) if $0<p<d/(d+1)$ and $s\in\SR$.
\end{proposition}

The same result for the spaces $B^s_{p,q}(\SRd)$ was proved by Triebel in \cite{triebel78}; 
see also \cite[Proposition 2.24]{triebel-bases}. Proposition \ref{P4} can be obtained from this and Theorem \ref{schauder} by suitable interpolation. 

Indeed, assertion (i) was already discussed in the paragraph following \eqref{limit}, so we restrict to $p\leq1$.
Assume next that $\sH_d$ is a basis for $F^s_{p,q}$ for some $d/(d+1)<p<1$ and $s>1$ or $s<d(1/p-1)$. 
By Theorem \ref{schauder}, $\sH_d$ is also a basis for $F^{s_0}_{p,q}$, for any $d(1/p-1)<s_0<1$.
By real interpolation, see e.g. \cite[Thm.\ 2.4.2(ii)]{Tr83}, for all $0<\theta<1$, the system $\sH_d$ will then be a basis of \[
\big(F^{s_0}_{p,q},F^{s}_{p,q}\big)_{\theta,q}=B^{s_\theta}_{p,q},\quad\mbox{with }s_\theta=(1-\theta)s_0+\theta s.\]
But when $\theta$ is close to 1 this would contradict Triebel's result. 
The remaining cases, $p=1$ and $p\geq d/(d+1)$ can be proved similarly using complex interpolation of $F$-spaces; see \cite[1.6.7]{triebel2}. 

We remark that, in the paper \cite{triebel78}, the failure of the Schauder basis property 
in the $B$-spaces is sometimes due to the fact that span$\,\sH_d$ fails to be dense in $B^s_{p,q}$. 
This is the case, for instance, in the region
\Be \label{dense}
(d-1)/d<p<1\mand \max\big\{1,d(1/p-1)\big\}<s<1/p;
\Ee
see \cite[Corollary 2]{triebel78}.  Here we show that also a quantitative bound holds, therefore ruling out the possibility that $\sH_d$ could be a basic sequence. 

\begin{proposition}\label{P5}
Let $0<q\leq \infty$, and $(p,s)$ be as in \eqref{dense}. Then, 
\[
\|\SE_N\|_{B^s_{p,q} \to B^s_{p,q}}\gtrsim 2^{(s-1)N}.
 \]
\end{proposition}
\begin{proof}
Let $\eta\in C^\infty_c(\SR^d)$ such that $\eta\equiv1$ on $[-2,2]^d$, and consider the Schwartz function $f(x)=x_1\,\eta(x)$. It suffices to show that
\Be
\big\|\SE_Nf\big\|_{B^s_{p,q}}\gtrsim 2^{(s-1)N}.
\label{P5_aux0}\Ee
Under \eqref{dense} we have $s>\sigma_p:=d(1/p-1)_+$. Assume first that $s<2$ (which is always the case if $d>1$).
Then we can use the equivalence of quasi-norms
\[
 \|g\|_{{B^s_{p,q}(\mathbb{R}^d)}}\approx \|g\|_p+\sum_{j=1}^d\Big(\int_0^1\frac{\|\Delta^2_{he_j}g\|^q_p}{h^{sq}}\,\frac{dh}h\Big)^{1/q}\,,\]
with the usual modification in the case $q=\infty$, 
see \cite[2.6.1]{triebel2}. In particular 
\Be
\big\|\SE_Nf\big\|_{B^s_{p,q}}\gtrsim \Bigg(\int_0^{2^{-N-1}}\frac{\big\|\Delta^2_{he_1}\big(\SE_Nf\big)\big\|^q_{L^p([0,1]^d)}}{h^{sq}}\,\frac{dh}h\Bigg)^{1/q}.\label{P5_aux}\Ee
Now, it is easily checked that, when $x\in[0,1)^d$, one has \[
\SE_Nf=\sum_{0\leq k<2^N}\tfrac{k+1/2}{2^N}\bbone_{[\frac{k}{2^N},\frac{k+1}{2^N})\times [0,1)^{d-1}},\]
and likewise, if we additionally assume $0<h<2^{-N-1}$, then \[
\Delta_{he_1}\big(\SE_Nf\big)=2^{-N-1}\sum_{k=1}^{2^N}\bbone_{[\frac{k}{2^N}-h,\frac{k}{2^N})\times [0,1)^{d-1}}.
\] 
and
\[
\Delta^2_{he_1}\big(\SE_Nf\big)=2^{-N-1}\sum_{k=1}^{2^N}\Big[\bbone_{[\frac{k}{2^N}-2h,\frac{k}{2^N}-h)\times [0,1)^{d-1}}-\bbone_{[\frac{k}{2^N}-h,\frac{k}{2^N})\times [0,1)^{d-1}}\Big].
\]
Therefore, 
\[
\|\Delta^2_{he_1}\SE_Nf\|_{L^p([0,1]^d)}= 2^{(N+1)(1/p-1)}\,h^{1/p},
\]
which, inserted into \eqref{P5_aux}, gives \eqref{P5_aux0}. If $d=1$ and $s\geq 2$, one applies a similar argument to the functions $\Delta^L_{he_1}(\SE_Nf)$ with $L=\lfloor s\rfloor +1$ and $h<2^{-N}/L$.
\end{proof}

By interpolation one obtains as well a quantitative bound for the relevant cases in Proposition \ref{P4}(ii).
\begin{corollary}\label{CP5}
Let $0<q\leq \infty$, $d/(d+1)<p<1$ and $1<s<1/p$. Then, for all $\e>0$, 
\Be
\|\SE_N\|_{F^s_{p,q} \to F^s_{p,q}}\gtrsim c_\e\,2^{(s-1-\e)N}.
 \label{CP5_aux}\Ee
\end{corollary}
\begin{proof}
If $d(1/p-1)<s_0<1$ and $\theta\in(0,1)$, then the real interpolation inequalities give\[
\big\|\SE_N\big\|_{F^{s_0}_{p,q} \to F^{s_0}_{p,q}}^{1-\theta}\,\big\|\SE_N\big\|_{F^s_{p,q} \to F^s_{p,q}}^{\theta}\,
\geq \,c_\theta\, \big\|\SE_N\big\|_{B^{s_{\theta}}_{p,q} \to B^{s_{\theta}}_{p,q}},
\]
with $s_\theta=(1-\theta)s_0+\theta s$. By Proposition \ref{P5} the right hand side is larger than 
a constant times $2^{N(s_\theta-1)}$, while by Corollary \ref{uniformbdcor} we have $\big\|\SE_N\big\|_{F^{s_0}_{p,q} \to F^{s_0}_{p,q}}\approx 1$. 
Choosing $\theta$ sufficiently close to 1 one derives \eqref{CP5_aux}.
\end{proof}

\bigskip

{\it Acknowledgments.}  The authors worked on this paper while participating in the 2016 
summer program in Constructive Approximation and Harmonic Analysis at the Centre de Recerca Matem\`atica at the Universitat Aut\`onoma de Barcelona, Spain. 
 They  would like to thank the organizers of the  program 
for providing a pleasant and fruitful research atmosphere. 
 We also thank the referee for various useful comments that have led to an improved version of this paper. Finally, T.U. thanks Peter Oswald for discussions concerning \cite{triebel78} and the results in \S\ref{optimal}.

\end{document}